\documentclass[12pt, reqno]{amsart}

\usepackage{amsmath,amssymb,amsthm,verbatim}

\newtheorem{lem}{Lemma}
\newtheorem{prop}{Proposition}

\newtheorem{thm}{Theorem}

\theoremstyle{definition}
\newtheorem{df}{Definition}

\begin{document}

\newcommand{\sinc}{\text{sinc}}

\title{Polyhyperbolic cardinal splines}

\author{Jeff Ledford}


\begin{abstract}
In this note we discuss solutions of differential equation $(D^2-\alpha^2)^{k}u=0$ on $\mathbb{R}\setminus\mathbb{Z}$, which we call hyperbolic splines.  We develop the fundamental function of interpolation and prove various properties related to these splines.

\end{abstract}

\maketitle

\section{Introduction}

Hyperbolic splines have been studied extensively, and are often called exponential splines \cite{UB}. In this article, we prove results similar to those found in \cite{baxter}, \cite{splines}, and \cite{RS}.  In each of these papers, a fundamental function of interpolation is studied.  

Our main result concerns band-limited functions and shows that as the degree of tends to infinity the polyhyperbolic spline interpolant tends to the target function.  Our methods require the use of the Fourier transform, for which we set the following convention.  If $g\in L_1(\mathbb{R})$, then its Fourier transform, written $\hat{g}$, is given by
\[
\hat{g}(\xi)= (2\pi)^{-1/2}\int_{\mathbb{R}}g(x)e^{-ix\xi}d\xi.
\]
This definition is extended to distributions when necessary.


\section{Polyhyperbolic Cardinal Splines}

\begin{df}
Given $\alpha>0$ and $k\in\mathbb{N}$, the set $H_\alpha^k$ of all \emph{k--hyperbolic cardinal splines} is given by
\begin{equation}\label{H}
H_\alpha^k=\left\{ f\in C^{2k-2}(\mathbb{R}): (D^2-\alpha^2)^{k}f=0, \text{ on } \mathbb{R}\setminus\mathbb{Z}         \right\}
\end{equation}
\end{df}
If $k\geq 1$, we say that $f\in H_\alpha^k$ is a \emph{polyhyperbolic cardinal spline}.
The justification for the name comes from the fact that $\{\cosh(\alpha \cdot), \sinh(\alpha\cdot)\}$ forms a fundamental set of solutions for the differential equation $(D^2-\alpha^2)f=0$.  In fact, it is not too hard to see that $\{(\cdot)^je^{\pm\alpha(\cdot)}:0\leq j< k\}$ forms a fundamental set of solutions for the equation $(D^2-\alpha^2)^kf=0$.

Most of this paper is analogous to the treatment given to polyharmonic cardinal splines in \cite{splines}.  We begin with what is called the fundamental solution in \cite{splines}.  The solution to the equation $(D^2-\alpha^2)^kE_k(t)=\delta(t)$ is given by
\begin{equation}\label{E}
E_k(t)=C_{\alpha,k} [e^{-\alpha|\cdot|}]^{*k}(t),
\end{equation}
where $f^{*k}$ is the convolution of $f$ with itself $k$ times and $C_{\alpha,k}$ is chosen so that the resulting multiple of the Dirac $\delta$--function is 1.  The Fourier transform of this function plays an important role in what follows, it is given by
\begin{equation}\label{hatE}
\widehat{E_k}(\xi)=(-1)^k(\xi^2+\alpha^2)^{-k}.
\end{equation}

From this function we develop a fundamental function of cardinal interpolation $L_k$, which is defined by its Fourier transform
\begin{equation}\label{hatL}
\widehat{L_k}(\xi)=(2\pi)^{-1/2}\dfrac{\widehat{E_k}(\xi)}{\sum_{j\in\mathbb{Z}}\widehat{E_k}(\xi-2\pi j)}.
\end{equation}

We will also be concerned with the related periodic distribution $\widehat{\Phi_k}$, defined by
\begin{equation}\label{hatPhi}
\widehat{\Phi_k}(\xi)=\widehat{E_k}(\xi)\widehat{L_k}(\xi).
\end{equation}
In order to study $L_k$ and $\Phi_k$ further, we make use of some basic analysis in the complex domain.
For $A\subset \mathbb{R}$, we define
\[
A_\epsilon=\{\zeta\in \mathbb{C}: \Re(\zeta)\in A  \text{ and } |\Im(\zeta)|\leq \epsilon  \}.
\]

We may now prove our first lemma.
\begin{lem}
There exists $\epsilon>0$ such that both $\widehat{\Phi_k}$ and $\widehat{L_k}$ have extensions to $\mathbb{R}_{\epsilon}$ which are analytic.
\end{lem}

\begin{proof}
We follow along the same lines as the $\alpha=0$ case which may be found in \cite{splines}.  We will henceforth assume that $\alpha>0$.  Set $q(\zeta)=-(\zeta^2+\alpha^2)$, where $\zeta=\xi+i\eta$.  We have
\[
(2\pi)^{-1/2}[\widehat{\Phi_k}(\xi)]^{-1}=\left\{1+ [q(\xi)]^{k}F(\xi)  \right\}[q(\xi)]^{-k},
\]
where
\[
F(\zeta)=\sum_{j\neq 0}[q(\zeta -2\pi j)]^{-k}.
\]
Since $\alpha>0$, $q^{-k}$ is analytic in $\mathbb{R}_\epsilon$ for any $\epsilon\in (0,\alpha)$.  It is also clear that $F$ is analytic on $(-\pi,\pi]_\epsilon$.  Now since $[q(\xi)]^kF(\xi)\geq 0$ for $\xi\in (-\pi,\pi]$, we can find $\epsilon>0$ (which may be smaller than our first choice) such that $1+[q(\zeta)]^kF(\zeta)$ has no zeros in $(-\pi,\pi]_\epsilon$.  This shows that $\widehat{\Phi_k}$ possesses an extension which is analytic in $(-\pi,\pi]_\epsilon$.  Using periodicity, we may extend this result to $\mathbb{R}_\epsilon$.
The analytic extension of $\widehat{L_k}$ to $\mathbb{R}_\epsilon$ is given by $\widehat{L_k}(\zeta)=\widehat{\Phi_k}(\zeta)[q(\zeta)]^{-k}$, which as the product of analytic functions is analytic.
\end{proof}

\begin{lem}
We have that
\begin{equation}\label{Phi}
\Phi_k(x)=\sum_{j\in\mathbb{Z}}a_j\delta(x-j),
\end{equation}
where the $a_j$'s depend on $k$ and $\alpha$.  Furthermore, there exists two constants $c,C>0$ such that for all $j\in\mathbb{Z}$,
\begin{equation}\label{aj}
|a_j|\leq C e^{-c|j|}.
\end{equation}
\end{lem}

\begin{proof}
Periodicity implies that \eqref{Phi} holds with
\[
a_j=(2\pi)^{-1/2} \int_{(-\pi,\pi]}\widehat{\Phi_k}(\xi)e^{ij\xi}d\xi.
\]
To see that \eqref{aj} holds, we use Lemma 1 to replace $(-\pi,\pi]$ with $\{ \zeta\in\mathbb{C}: \Re(\zeta)\in (-\pi,\pi] \text{ and } \Im(\zeta)=\text{sgn}(j)c \}$,  where $0<c<\epsilon$.  We have
\[
|a_j|\leq (2\pi)^{-1/2}e^{-c |j|}\int_{(-\pi,\pi]}|\widehat{\Phi_k}(\xi)|d\xi\leq Ce^{-c|j|}.
\]
\end{proof}

\begin{prop}
Let $L_k$ be defined by its Fourier transform \eqref{hatL}, where $k\in\mathbb{N}$.  Then $L_k$ has the following properties:
\begin{itemize}
\item[(i)] $L_k$ is a $k$--hyperbolic cardinal spline.
\item[(ii)] For all $j\in\mathbb{Z}$, $L_k(j)=\delta_{0,j}$, where $\delta$ is the Kronecker $\delta$ function.
\item[(iii)] There are constants $c,C>0$ (which depend on $k$ and $\alpha$) such that
\[
|L_k(x)|\leq C e^{-c|x|}
\]
for all $x\in \mathbb{R}$.
\item[(iv)] $L_k$ has the following representation in terms of $E_k$:
\[
L_k(x)=\sum_{j\in\mathbb{Z}}a_jE(x-j)=\Phi_k*E_k(x),
\]
where $\Phi_k$ is the function whose Fourier transform is defined by \eqref{hatPhi} and the $a_j$'s are those from Lemma 2.
\end{itemize}
\end{prop}

\begin{proof}
(i) We see that $(D^2-\alpha^2)^kL_k=\Phi_k$, from examining the equation in the Fourier domain. Now Lemma 2 shows that $L_k\in H_\alpha^k$.

(ii) We have
\begin{align*}
L_k(j)=&(2\pi)^{-1/2}\int_{\mathbb{R}}\widehat{L_k}(\xi)e^{ij\xi}d\xi=(2\pi)^{-1/2}\int_{-\pi}^{\pi}\sum_{l\in\mathbb{Z}}\widehat{L_k}(\xi-2\pi l)e^{ij\xi}d\xi\\
=&(2\pi)^{-1}\int_{-\pi}^{\pi}e^{ij\xi}d\xi = \delta_{0,j}.
\end{align*}
The interchange of the sum and in the integral in the ``periodization trick" is justified by the decay of $\widehat{E_k}$ and the periodicity of the denominator of $\widehat{L_k}$.  That the sum is $(2\pi)^{-1/2}$ is clear from \eqref{hatL} and Lemma 1.

(iii) An argument completely analogous to the one given in Lemma 2 show that this is true.

(iv) This is clear from \eqref{hatL} and \eqref{hatPhi}.
\end{proof}

We turn briefly to the data that we wish to interpolate.
\begin{df}
For $\beta>0$, let $Y^\beta$ be the set of all sequences $b=\{b_j:j\in\mathbb{Z}\}$ which satisfy
\begin{equation}\label{Y}
|b_j|\leq C(1+|j|)^{\beta},
\end{equation}
for some constant $C>0$.
\end{df}

Our next result shows what to expect when we interpolate data of this type. 
\begin{prop}
Suppose that $b=\{b_j:j\in\mathbb{Z}\}$ is a sequence of polynomial growth and define the function $f_b$ by
\begin{equation}\label{fb}
f_b(x)=\sum_{j\in\mathbb{Z}}b_jL_k(x-j).
\end{equation}
The following are true:
\begin{itemize}
\item[(i)] The expansion \eqref{fb} converges absolutely and uniformly in every compact subset of $\mathbb{R}$. 
\item[(ii)] The function $f_b$ is a $k$--hyperbolic spline and $f_b(j)=b_j$ for all $j\in\mathbb{Z}$.
\item[(iii)] If $b\in Y^\beta$, then $f_b$ satisfies
\[
|f_b(x)|\leq C(1+|x|)^{\beta},
\]
for some $C>0$.
\item[(iv)] If $g\in\text{\emph span}(\{(\cdot)^je^{\pm\alpha(\cdot)}:0\leq j< k\})$ and $b=\{g(j):j\in\mathbb{Z}\}    $, then $f_b(x)=g(x)$ for all $x\in\mathbb{R}$.
\end{itemize}
\end{prop}

\begin{proof}
(i) and (ii) are straightforward consequences of Proposition 1.

(iii)  Let $m_x\in\mathbb{Z}$ satisfy $m_x\leq x < m_x+1$.  We have
\begin{align*}
&|f_b(x)|\leq \sum_{j\in\mathbb{Z}}|b_jL_k(x-j)|\leq C\sum_{j\in\mathbb{Z}}(1+|j|)^\beta e^{-c|x-k|}\\
=&C(1+|m_x|)^\beta e^{-c|x-m_x|}\left\{ 1+ \sum_{j\neq m_x}\left( \dfrac{1+|j|}{1+|m_x|}  \right)^\beta e^{c|x-m_x|-c|x-j|}    \right\}\\
=:& C(1+|m_x|)^\beta e^{-c|x-m_x|}\{1 + A+ B  \},
\end{align*}
\end{proof}
where
\[
A= \sum_{j>m_x}\left( \dfrac{1+|j|}{1+|m_x|}  \right)^\beta e^{c|x-m_x|-c|x-j|},
\]
and
\[
B=\sum_{j<m_x}\left( \dfrac{1+|j|}{1+|m_x|}  \right)^\beta e^{c|x-m_x|-c|x-j|}.
\]
The estimates being similar, we show only the estimate for $A$:
\begin{align*}
A=&e^{c|x-m_x|}\sum_{j=m_x+1}^{\infty}\left( \dfrac{1+|j|}{1+|m_x|}  \right)^\beta e^{-c|x-j|}\\
\leq & e^c \sum_{k=1}^{\infty}\left( \dfrac{1+|m_x+k|}{1+|m_x|} \right)^\beta e^{-c(k+m_x-x)}\\
\leq & e^{c-c(m_x-x)}\sum_{k=1}^{\infty}\left( 1 + |k|\right)^\beta e^{-ck}\\
\leq & e^{2c}\sum_{k=1}^{\infty}\left( 1 + |k|\right)^\beta e^{-ck} <\infty.
\end{align*}
Here we have used the triangle inequality and reindexed the sum.  In the second inequality we have used that $\beta>0$.  Finally, the last inequality comes from the fact that $m_x\leq x < m_x+1$.  Combining this estimate with the analogous estimate for $B$ yields the desired result.

(iv) This is a consequence of the linear independence of the set $\{(\cdot)^je^{\pm\alpha(\cdot)}:0\leq j< k\}\subset H_{\alpha}^{k}$.  Since $f_b(j)$ and $g(j)$ conincide on $2k$ consecutive integers, say $0\leq j\leq 2k-1$, we have $f_b(x)=g(x)$ on the interval $[0,2k-1]$.  Repeating this argument in both directions shows that $f_b(x)=g(x)$ for all $x\in\mathbb{R}$.

We end this note with a result concerning the behavior of polyhyperbolic splines as their degree tends to infinity.

For this we recall that the classical Paley-Wiener space $PW$, is defined by
\[
PW=\{ g\in L_2(\mathbb{R}): \hat{g}(\xi)=0 \text{ a.e. } \xi\notin [-\pi,\pi]   \}.
\]
A function $g\in PW$ is often called \emph{band-limited}.  The Paley-Wiener theorem shows that band-limited functions are entire functions of exponential type whose restrictions to the real line are in $L_2(\mathbb{R})$.  

In light of Proposition 2, we may interpolate a function $g\in PW$ with a $k$--hyperbolic spline as follows
\begin{equation}\label{Ik}
I_k[g](x)=\sum_{j\in\mathbb{Z}}g(j)L_k(x-j).
\end{equation}

\begin{thm}
Suppose that $g\in PW$.  The function $I_k[g]$ defined in \eqref{Ik} satisfies
\begin{enumerate}
\item[(i)] \begin{equation}\label{thmi}
\lim_{k\to\infty} \| g- I_k[g] \|_{L_2(\mathbb{R})}=0, \text{ and}
\end{equation} 
\item[(ii)] \begin{equation}\label{thmii}
\lim_{k\to\infty} | g(x)- I_k[g](x) |=0, \text{ uniformly on } \mathbb{R}.
\end{equation}
\end{enumerate}
\end{thm}

\begin{proof}
We begin by noting that if $y=\{ y_j: j\in\mathbb{Z}  \}\in \ell_p$ for $1\leq p \leq \infty$ then $f_y\in L_p(\mathbb{R})$, where $f_y$ is defined by \eqref{fb}.  Since $L_k$ has exponential decay, the estimates of the form
\begin{align*}
\| f_y  \|_{L_1(\mathbb{R})}\leq & C_k \| y \|_{\ell_1}, \text{ and}\\
\| f_y \|_{L_\infty(\mathbb{R})} \leq & C'_k \| y \|_{\ell_\infty}
\end{align*}
are easily established.  Now the Riesz-Thorin interpolation theorem provides a constant $C_{k,p}>0$ such that 
\[
\| f_y  \|_{L_p(\mathbb{R})}\leq  C_{k,p} \| y \|_{\ell_p}.
\]
It is well known that if $g\in PW$, then $\{g(j):j\in\mathbb{Z}\}\in\ell_2$.  Hence we can see that
\[
\| I_k[g]  \|_{L_2(\mathbb{R})}\leq C_{k}\|\{g(j):j\in\mathbb{Z}\}  \|_{\ell_2}
\]
for some constant $C_k>0$.  This allows us to make use of the Fourier transform.  We have
\[
\widehat{I_k[g]}(\xi)= \widehat{L_k}(\xi)\sum_{j\in\mathbb{Z}}g(j)e^{ij\xi}.
\]
In particular, for $\xi\in (-\pi, \pi]$, $\widehat{I_k[g]}(\xi)= \sqrt{2\pi}\widehat{L_k}(\xi)\hat{g}(\xi)$.
Now we have the following estimates using Plancherel's identity and the fact that $g$ is band-limited:
\begin{align*}
&\| g- I_k[g] \|_{L_2(\mathbb{R})}^2  = \| \hat{g}- \widehat{I_k[g]} \|_{L_2(\mathbb{R})}\\
&= \int_{-\pi}^{\pi}|\hat{g}(\xi)|^2\left(1-\sqrt{2\pi}\widehat{L_k}(\xi)  \right)^2d\xi + \sum_{\ell\neq 0}\sqrt{2\pi}\int_{-\pi}^{\pi}|\hat{g}(\xi)\widehat{L_k}(\xi-2\pi\ell)|^2d\xi\\
&=\sqrt{2\pi} \int_{-\pi}^{\pi} |\hat{g}(\xi)|^2 \left(\sum_{\ell\neq 0}\widehat{L_k}(\xi-2\pi\ell)\right)^2d\xi\\
&\qquad +\quad \sqrt{2\pi}\int_{-\pi}^{\pi}|\hat{g}(\xi)|^2 \sum_{\ell\neq 0}|\widehat{L_k}(\xi-2\pi\ell)|^2 d\xi.\\
&\leq 2\sqrt{2\pi} \int_{-\pi}^{\pi} |\hat{g}(\xi)|^2\left( \sum_{\ell\neq 0}\widehat{L_k}(\xi-2\pi\ell)  \right)^2d\xi   
\end{align*}
The second equality follows from periodization, while the third follows from Tonelli's theorem and manipulating \eqref{hatL}.  The final inequality follows from the positivity of the terms in the sum.  In light of \eqref{hatE}, we have the following elementary estimate for $\xi\in [-\pi,\pi]$ and $\ell\neq 0$:
\[
|L_k(\xi-2\pi\ell)|\leq (2\pi)^{-1/2}\left(\dfrac{\pi^2+\alpha^2}{(2|\ell|-1)^2\pi^2+\alpha^2}  \right)^k.
\]
This allows us to use the dominated convergence theorem since the corresponding series is bounded.  For $|\xi|<\pi$, the estimate above works with $\xi^2$ replacing $\pi^2$ in the numerator, thus each term in parentheses is strictly less that 1, hence tends to 0 as $k\to\infty$.  This completes the proof of (i).
The proof of (ii) follows from (i) together with the Cauchy-Schwarz inequality.  Using the inversion formula and the triangle inequality, we have
\begin{align*}
| g(x)- I_k[g](x) | & =(2\pi)^{-1/2} \left|\int_{\mathbb{R}} \left(\hat{g}(\xi)-\widehat{I_k[g]}(\xi)\right)e^{ix\xi}d\xi \right|\\ 
& \leq (2\pi)^{-1/2}\int_{\mathbb{R}} \left|\hat{g}(\xi)-\widehat{I_k[g]}(\xi)\right|d\xi.
\end{align*}
Now we may use periodization and the Cauchy-Schwarz inequality and reasoning similar to that used in (i) to see that
\[
\left| g(x)-I_k[g](x)  \right| \leq  2\int_{-\pi}^{\pi} |\hat{g}(\xi)|^2\left(\sum_{\ell\neq 0}\widehat{L_k}(\xi-2\pi\ell)   \right)^2d\xi,
\]
which tends to $0$ independently of $x\in\mathbb{R}$. 
\end{proof}

\end{document}